\documentclass{elsarticle}

\usepackage[english]{babel}
\usepackage{dsfont} 
\usepackage{graphicx}
\usepackage{color}
\usepackage[hypertexnames=false]{hyperref}
\usepackage{enumerate}
\usepackage{amssymb,empheq} 
\usepackage{amsthm} 
\usepackage{thmtools}
\usepackage{cleveref}
\usepackage{xcolor}

\allowdisplaybreaks
\newtheorem{thm}{Theorem}[section]
\newtheorem{lem}[thm]{Lemma}
\newtheorem{prop}[thm]{Proposition}
\newtheorem{cor}[thm]{Corollary}

\theoremstyle{definition}

\newtheorem{dfn}[thm]{Definition}

\newtheorem{exm}[thm]{Example}

\theoremstyle{remark}
\newtheorem{rem}[thm]{Remark}

\newcommand{\exmsymbol}{\hfill$\circ$}

\newcommand{\cset}{\mathds{C}}

\newcommand{\nset}{\mathds{N}}

\newcommand{\rset}{\mathds{R}}

\newcommand{\conv}{\mathrm{conv}\,}

\newcommand{\diff}{\mathrm{d}}

\newcommand{\pos}{\mathrm{Pos}}

\newcommand{\supp}{\mathrm{supp}\,}

\newcommand{\inter}{\mathrm{int}\,}
\newcommand{\id}{\mathrm{id}}
\newcommand{\one}{\mathds{1}}

\newcommand{\cone}{\mathrm{cone}\,}
\newcommand{\Ev}{\mathrm{Ev}}

\newcommand{\cB}{\mathcal{B}}

\newcommand{\cS}{\mathcal{S}}

\newcommand{\cV}{\mathcal{V}}

\newcommand{\fg}{\mathfrak{g}}

\author{Philipp J.\ di Dio}
\address{Department of Mathematics and Statistics, University of Konstanz, Universit\"atsstra{\ss}e 10, D-78464 Konstanz, Germany}
\address{Department of Computer and Information Science, University of Konstanz, Universit\"atsstra{\ss}e 10, D-78464 Konstanz, Germany}
\address{Zukunftskolleg, Universtity of Konstanz, Universit\"atsstra{\ss}e 10, D-78464 Konstanz, Germany}
\address{philipp.didio@uni-konstanz.de}

\journal{ArXiv}

\title{Making Non-Negative Polynomials into Sums of Squares}

\begin{document}

\begin{abstract}
We study linear operators $T:\rset[x_1,\dots,x_n]\to\rset[x_1,\dots,x_n]$, especially for the purpose to move sets $S\subseteq\rset[x_1,\dots,x_n]$ into cones $C\subseteq\rset[x_1,\dots,x_n]$: $TS\subseteq C$.
We develop the theory of (semi-)groups of operators $(e^{tA})_{t\in\rset}$ on $\rset[x_1,\dots,x_n]$, which requires techniques from regular Fr\'echet Lie groups.
We study the special case of making non-negative polynomials $\pos(K)_{\leq 2d}$ with $K\subseteq\rset^n$ and $\inter K\neq \emptyset$ into sums of squares: $\tilde{T}\pos(K)_{\leq 2d}\subseteq \sum\rset[x_1,\dots,x_n]_{\leq d}^2$.
With $N:=\dim\rset[x_1,\dots,x_n]_{\leq 2d} = \binom{n+2d}{n}$, for $\tilde{T}$, a memory of at most $2N+1$ is required.
Matrix multiplications $\tilde{T}M$, $M\tilde{T}$, $\tilde{T}^{-1}M$, and $M\tilde{T}^{-1}$ of $\tilde{T}$ with any $M\in\rset^{N\times N}$ require at most $4N^2+1$ operations.
Transformations $\tilde{T}v$ and $\tilde{T}^{-1}v$ of vectors $v\in\rset^N$ require at most $4N+1$ operations.
Calculating $\tilde{T}^{-1}$ of $\tilde{T}$ requires only \emph{one} (!) operation.
\end{abstract}


\maketitle


\section{Introduction}
\label{sec:intro}

Let $n\in\nset$ and let $K\subseteq\rset^n$ be closed.
Denote by $\rset[x_1,\dots,x_n]$ the ring of polynomials in $n$ variables $x_1$, \dots, $x_n$ with real coefficients.
The cone
\[\pos(K) := \big\{ f\in\rset[x_1,\dots,x_n] \,\big|\, f(x)\geq 0\ \text{for all}\ x\in K\big\}\]
of polynomials which are non-negative on $K$ is very well-studied in real algebraic geometry, see e.g.\ \cite{schmud91,putina93,marshallPosPoly,bochnak98,prestelPosPoly,
schmudMomentBook,powersPositivityRealPolynomials,scheidererRealAlgebraicGeometry}.

Dual to $\pos(K)$ is the $K$-moment problem, i.e., when is a linear functional
\[L:\rset[x_1,\dots,x_n]\to\rset\]
represented by a measure $\mu$ with $\supp\mu\subseteq K$:
\[L(f) = \int f(x)~\diff\mu(x)\]
for all $f\in\rset[x_1,\dots,x_n]$.
This is also a very well-studied and still active area of research, see e.g.\ \cite{akhiezClassical,kreinMarkovMomentProblem,schmud91,marshallPosPoly,lauren09,schmudMomentBook}.

Non-negative polynomial $\pos(K)$ and $K$-moment sequences are hard to handle, since there are non-negative polynomials which are not sums of squares \cite{motzkin65} and there are linear functionals $L$ which are square positive, i.e., $L(f^2)\geq 0$ for all $f\in\rset[x_1,\dots,x_n]$, but which are not moment functionals \cite{schmud79}.

It is therefore surprising that linear operators (maps)
\begin{equation}\label{eq:linOpT}
T:\rset[x_1,\dots,x_n]\to\rset[x_1,\dots,x_n]
\end{equation}
are much less studied, especially when $T$ preservers non-negativity on $K$:
\begin{equation}\label{eq:KposPres}
T\pos(K)\subseteq\pos(K),
\end{equation}
see e.g.\ \cite{guterman08,netzer10,borcea11,curtoHeat22,didio24posPresConst,didio25KPosPresGen,
didio25hadamardLanger,didio26matrix,curtoHeat23,langer26operator}.
Operators which fulfill (\ref{eq:KposPres}) are called \emph{$K$-positivity preserver}.
See especially \cite{didio24posPresConst,didio25KPosPresGen,didio25hadamardLanger,
didio26matrix} for more recent progress in the theory of $K$-positivity preservers.
One of the main key aspects was the investigation of maps of the form
\[T_t = e^{tA}:\rset[x_1,\dots,x_n]\to\rset[x_1,\dots,x_n]\]
for some $A:\rset[x_1,\dots,x_n]\to\rset[x_1,\dots,x_n]$.

It has long been known that every operator (\ref{eq:linOpT}) has the form
\begin{equation}\label{eq:canonicalRepr}
T = \sum_{\alpha\in\nset_0^n} q_\alpha\cdot\partial^\alpha
\end{equation}
with unique polynomial coefficients $q_\alpha\in\rset[x_1,\dots,x_n]$.
This was already known to Hilbert and a proof can be found e.g.\ \cite{netzer10}.
The form (\ref{eq:canonicalRepr}) is called the \emph{canonical representation}.
See \cite{guterman08,netzer10,borcea11} for more on the topic.

With (\ref{eq:canonicalRepr}), for any $y\in\rset^n$, we define
\[T_y := \sum_{\alpha\in\nset_0^n} q_\alpha(y)\cdot\partial^\alpha.\]
$K$-Positivity preserver (\ref{eq:KposPres}) are deeply connected to $K$-moment sequences, see \Cref{dfn:KmomSeq} for $K$-moment sequences.

$K$-positivity preserver are fully characterized by the following result.
For $K=\rset^n$ this was already shown in \cite{borcea11} and an attempt for general $K\subseteq\rset^n$ was made in \cite{netzer10}.
The complete characterization was then proven in \cite{didio25KPosPresGen}.

\begin{thm}[{\cite[Thm.\ 3.5]{didio25KPosPresGen}}]\label{thm:KposPres}
Let $n\in\nset$, let $K\subseteq\rset^n$ be closed, and let
\[T = \sum_{\alpha\in\nset_0^n} q_\alpha\cdot\partial^\alpha:\rset[x_1,\dots,x_n]\to\rset[x_1,\dots,x_n]\]
be linear, $q_\alpha\in\rset[x_1,\dots,x_n]$ for all $\alpha\in\nset_0^n$.
Then the following are equivalent:
\begin{enumerate}[(i)]
\item $T$ is a $K$-positivity preserver.

\item For all $y\in K$, $(\alpha!\cdot q_\alpha(y))_{\alpha\in\nset_0^n}$ is a $(K-y)$-moment sequence.
\end{enumerate}
\end{thm}

In the current investigation we are not just interested in $K$-positivity preservers, but in linear maps $T$ with
\[T\pos(K) \subseteq \sum\rset[x_1,\dots,x_n]^2,\]
i.e., we want to transform non-negative polynomials $\pos(K)$ into sums of squares $\sum\rset[x_1,\dots,x_n]^2$.
More generally, we are interested in using $T$ of the form $T = e^{A}$ such that a given set $S\subseteq\rset[x_1,\dots,x_n]$ is moved into some cone $C\subseteq\rset[x_1,\dots,x_n]$.
This allows then for two (possible) applications:
\begin{enumerate}[\quad (a)]
\item Let
\[F:\rset[x_1,\dots,x_n]\to\rset,\]
be a given (not necessarily linear) function(al), let $S\subseteq\rset[x_1,\dots,x_n]$, and let $T$ be a bijective (linear) operator (\ref{eq:linOpT}).
Then, clearly,
\[\min_{f\in S} F(f) = \min_{f\in TS} (F\circ T^{-1})(f).\]

\item Let
\[F:\rset[x_1,\dots,x_n]\to\rset[x_1,\dots,x_n],\]
be a given (not necessarily linear) operator.
Then one seeks for a (linear) operator $T$ such that
\[1\approx \kappa(F\circ T) \ll \kappa(F),\]
where $\kappa$ is the \emph{condition number}.
This is called \emph{preconditioning}.
\end{enumerate}
In (a) one is interested in moving the set $S$ to a nicer set $TS$, e.g., a subset of sums of squares.
In (b) one is interested in increasing the numerical stability of a problem (lowering the condition number $\kappa$ by preconditioning), such that larger problems can be handled with higher precision or are numerically more stable.
In both cases we have to pay the additional price to calculate
\[T^{-1},\qquad F\circ T,\qquad\text{or}\qquad F\circ T^{-1}.\]
These are infinite versions of matrix multiplications.
With
\[T:\rset[x_1,\dots,x_n]_{\leq d}\to\rset[x_1,\dots,x_n]_{\leq d}\]
for some $d\in\nset_0$, these are the usual matrix multiplications and inversion.
The theory we are developing here therefore also has to account for computational costs, when we want any hope for an application.
As we will see, the approach via operators
\[T_t = e^{tA}\]
fulfills this requirement.

Matrix multiplication and inversion are \emph{very expensive} calculations and the matrix inversion is even numerically unstable.
For matrix multiplication in $\rset^{N\times N}$, the \emph{schoolbook algorithm} requires at most
\[2N^3-N^2\]
operations, while the \emph{Strassen's algorithm} requires at most
\[4.7\cdot N^{\log_2 7} \qquad\text{with}\qquad \log_2 7 \approx 2.8073549\]
operations \cite[Fact 2, p.\ 355]{strassen65}.
Matrix inversion is possible with at most $5.64\cdot N^{\log_2 7}$ operations \cite[Fact 4, p.\ 356]{strassen65}.
Subsequent investigations on matrix multiplication in \cite{alman25} lowered the complexity bound $O(N^\omega)$ to
\[\omega \leq 2.371339.\]
While the schoolbook and Strassen's algorithms are used in practice, the subsequent improvements (if explicitly known) have large hidden constants and are therefore \emph{galactic algorithms}.
Till to today, the lower bound on $\omega$ is the trivial lower bound
\[2\leq\omega.\]

By \Cref{rem:complexity}, we therefore have to emphasize, that the application of the transformation $T$ in \Cref{thm:main} is not only $O(N^2)$, but it requires at most
\[4N^2+1\]
operations, i.e., it is an extremely efficient algorithm compared to the expensive matrix multiplication to perform the transformation $T$.
The matrices $\tilde{T}$ are so special and efficient, that the complexity meets even the trivial lower bound $2\leq\omega$.
The efficiency of the transformations in \Cref{thm:main} is also reflected in the facts, that saving $\tilde{T}\in\rset^{N\times N}$ requires a memory of only $2N+1$ and \emph{calculating the inverse $\tilde{T}^{-1}$ requires only a single operation}!
The transformation
\[Tp \qquad\text{or}\qquad T^{-1}p\]
of a single polynomial $p$ requires at most
\[4N+1\]
operations.
Therefore, readers, who are more interested in applications and complexity, might now see the true power of the abstract theory developed in Sections \ref{sec:fg} to \ref{sec:KposGen}, since requiring a single operation to calculate $\tilde{T}^{-1}\in\rset^{N\times N}$, only $4N+1$ operations for the transformation of a single polynomial, and only $4N^2+1$ operations for the transformation of a map is \emph{computationally very efficient}.

\paragraph{Structure of the Paper}
In \Cref{sec:prelim} we give the preliminaries to this paper to make it as self-contained as possible.
In \Cref{sec:fg} we define and fully characterize the set $\fg$ of all linear maps
\[A:\rset[x_1,\dots,x_n]\to\rset[x_1,\dots,x_n]\]
such that $e^{tA}:\rset[x_1,\dots,x_n]\to\rset[x_1,\dots,x_n]$ for all $t\in\rset$.
In \Cref{sec:fgAlgebras} we further investigate the set $\fg$ and we characterize several Lie algebras in $\fg$ but also show that $\fg$ is itself not a Lie algebra.
Because of the connection of $K$-positivity preservers to $K$-moment sequences (\Cref{thm:KposPres}), in \Cref{sec:stochel} we give a three-line-proof of Stochel's Theorem (\Cref{thm:stochel}).
In \Cref{sec:KposGen} we then characterize generators of $K$-positivity preserving semi-groups.
In \Cref{sec:nonnegativeSoS} we apply the theory developed in Sections \ref{sec:fg} to \ref{sec:KposGen} to make non-negative polynomials into sums of squares.
Since $(e^{tA})_{t\geq 0}$ is a homogeneous Markov process on $\rset[x_1,\dots,x_n]$, in \Cref{sec:nonMarkov} we give an example of a process which is not a homogeneous Markov process on $\rset[x_1,\dots,x_n]$ to demonstrate further directions in which the theory can be developed and applied.

\section{Preliminaries}
\label{sec:prelim}

Let
\[n\in\nset := \{1,2,3,\dots\},\quad \nset_0 := \{0,1,2,\dots\}, \quad\text{and}\quad \alpha = (\alpha_1,\dots,\alpha_n)\in\nset_0^n.\]
We set
\[x^\alpha := x_1^{\alpha_1}\cdots x_n^{\alpha_n},\quad \partial^\alpha := \partial_1^{\alpha_1}\cdots\partial_n^{\alpha_n},\quad \alpha! := \alpha_1!\cdots \alpha_n!,\quad |\alpha| := \alpha_1 + \dots + \alpha_n.\]

\subsection{Moments}

In \Cref{thm:KposPres} we have seen the connection of $K$-positivity preserver to $K$-moment sequences.
Hence, for the readers convenience we repeat the definition of $K$-moment sequence.

\begin{dfn}\label{dfn:KmomSeq}
Let $n\in\nset$ and let $K\subseteq\rset^n$ be closed.
A real sequence
\[s = (s_\alpha)_{\alpha\in\nset_0^n}\]
is called a \emph{$K$-moment sequence}, if there exists a measure $\mu$ on $\rset^n$ with
\[\supp\mu\subseteq K \qquad\text{and}\qquad s_\alpha = \int x^\alpha~\diff\mu(x)\]
for all $\alpha\in\nset_0^n$.
\end{dfn}

\begin{dfn}\label{dfn:riesz}
Let $n\in\nset$ and let $s = (s_\alpha)_{\alpha\in\nset_0^n} \in\cset^{\nset_0^n}$.
We define the \emph{Riesz functional}
\[L_s:\cset[x_1,\dots,x_n]\to\cset \qquad\text{by}\qquad L_s(x^\alpha) = s_\alpha\qquad \text{for all}\ \alpha\in\nset_0^n\]
with linearly extension to all $\cset[x_1,\dots,x_n]$.
\end{dfn}

\subsection{Topology}\label{subsec:top}

Since we have to work with infinite sums and limits, we have to introduce the Fr\'echet and the LF-topology for our purposes and give for our purposes the two most important examples.

A \emph{Fr\'echet space} is a metrizable, complete, and locally convex space.
Our guiding example is the following.

\begin{exm}[see e.g.\ {\cite[pp.\ 91--92, Exm.\ III]{treves67}}]
Let $n\in\nset$.
Equip
\[\cset^{\nset_0^n} := \big\{s = (s_\alpha)_{\alpha\in\nset_0^n} \,\big|\, s_\alpha\in\cset\ \text{for all}\ \alpha\in\nset_0^n\big\}\]
with the semi-norms
\[|s|_d := \max_{\alpha\in\nset_0^n:|\alpha|\leq d} |s_\alpha|\]
for all $d\in\nset_0$.
Then $\cset^{\nset_0^n}$ is a Fr\'echet space.
\exmsymbol
\end{exm}

For the exact definition of an LF-space see \cite[p.\ 126]{treves67}.
For our purposes it is sufficient to have the following example.

\begin{exm}[see e.g.\ {\cite[p.\ 130, Exm.\ I]{treves67}}]
Let $n\in\nset$.
Then
\[\cset[x_1,\dots,x_n]\]
is a LF-space with defining sequence
\[E_d := \cset[x_1,\dots,x_n]_{\leq d}.\]
To understand the LF-topology of $\cset[x_1,\dots,x_n]$ it is sufficient to know that a sequence $(p_i)_{i\in\nset_0}$ in $\cset[x_1,\dots,x_n]$ converges to some $p\in\cset[x_1,\dots,x_n]$ if and only if there exists a $D\in\nset_0$ with
\[\sup_{i\in\nset_0} \deg p_i \leq D < \infty \qquad\text{and}\qquad p_i\xrightarrow{i\to\infty} p\ \text{in}\ \cset[x_1,\dots,x_n]_{\leq D}. \tag*{$\circ$}\]
\end{exm}

The spaces
\[\cset^{\nset_0^n} \qquad\text{and}\qquad \cset[x_1,\dots,x_n]\]
are dual to each other.
The dual pairing $\langle\,\cdot\,\,,\,\cdot\,\rangle$ is given by the Riesz functional $L_s$ from \Cref{dfn:riesz} by
\[\langle s,p\rangle := L_s(p)\]
for all $s\in\cset^{\nset_0^n}$ and $p\in\cset[x_1,\dots,x_n]$.
For $\varepsilon > 0$, a curve
\[c: (-\varepsilon,\varepsilon)\to\rset[x_1,\dots,x_n], \quad t\mapsto c(t) = \sum_{\alpha\in\nset_0^n} c_\alpha(t)\cdot x^\alpha\]
is $C^k$ for some $k\in\nset_0$, if and only if, for all $s\in\cset^{\nset_0^n}$,
\[\langle s,g(t)\rangle = L_s(g(t)): (-\varepsilon,\varepsilon)\to\rset \qquad\text{is in}\qquad C^k((-\varepsilon,\varepsilon),\rset).\]

We gave in (\ref{eq:canonicalRepr}) the canonical representation of linear operators
\[T:\cset[x_1,\dots,x_n]\to\cset[x_1,\dots,x_n].\]
Via the duality
\[\langle T^* s, p\rangle = \langle s, Tp\rangle\]
for all $s\in\cset^{\nset_0^n}$ and $p\in\cset[x_1,\dots,x_n]$ we get linear operators $S = T^*$ with
\[S:\cset^{\nset_0^n}\to\cset^{\nset_0^n}\]
on sequences.
Clearly,
\[T^{**} = T \qquad\text{and}\qquad S^{**} = S.\]
Taking $e_\alpha := (\delta_{\alpha,\beta})_{\beta\in\nset_0^n}$ as basis of $\cset^{\nset_0^n}$, operators $S$ on sequences are represented by infinite matrices such that in every row there are only finitely many non-zero entries.
To see this, we have that the projection to the $\alpha$-coordinate
\[\cset^{\nset_0^n}\to\cset,\; s\mapsto (Ss)_\alpha \qquad\text{is in}\qquad (\cset^{\nset_0^n})^* \cong \cset[x_1,\dots,x_n]\]
for all $\alpha\in\nset_0^n$ and hence there exists a polynomial $p_\alpha\in\cset[x_1,\dots,x_n]$ with
\[ (Ss)_\alpha = \langle s, p_\alpha\rangle\]
for all $s\in\cset^{\nset_0^n}$.

For more on topological vector spaces see e.g.\ \cite{treves67,koethe69TopVecSp1,koethe79TopVecSp2,schaef99}.

\subsection{Regular Fr\'echet Lie Groups}

We are interested all linear operators $A:\rset[x_1,\dots,x_n]\to\rset[x_1,\dots,x_n]$ such that
\[(e^{tA}:\rset[x_1,\dots,x_n]\to\rset[x_1,\dots,x_n])_{t\in\rset}\]
is a group.
We already made the connection to regular Fr\'echet Lie groups in \cite{didio24posPresConst} and \cite{didio25KPosPresGen}.
Hence, we repeat its definition here.

\begin{dfn}[see e.g.\ {\cite[p.\ 63, Dfn.\ 1.1]{omori97}}]\label{dfn:frechetLieGroup}
We call $(G,\,\cdot\,)$ a (\emph{regular}) \emph{Fr\'echet Lie group} if the following conditions are fulfilled:
\begin{enumerate}[(i)]
\item $G$ is an infinite dimensional smooth Fr\'echet manifold.

\item $(G,\,\cdot\,)$ is a group.

\item The map $G\times G\to G$, $(A,B)\mapsto A\cdot B^{-1}$ is smooth.

\item The \emph{Fr\'echet Lie algebra} $\fg$ of $G$ is isomorphic to the tangent space $T_e G$ of $G$ at the unit element $e\in G$.

\item $\exp:\fg\to G$ is a smooth mapping such that
\[\left.\frac{\diff}{\diff t}\right|_{t=0} \exp(t u) = u\]
holds for all $u\in\fg$.

\item The space $C^1(G,\fg)$ of $C^1$-curves in $G$ coincides with the set of all $C^1$-curves in $G$ under the Fr\'echet topology.
\end{enumerate}
\end{dfn}

For more on infinite dimensional manifolds, differential calculus, Lie groups, and Lie algebras see e.g.\ \cite{leslie67,omori74,omori97,schmed23}.

\section{The Set $\fg$ of Generators}
\label{sec:fg}

Our main object of interest in this section is the following.

\begin{dfn}\label{dfn:fg}
Let $n\in\nset$ and let $\varepsilon>0$.
We define
\begin{multline*}
\fg := \Big\{A:\rset[x_1,\dots,x_n]\to\rset[x_1,\dots,x_n]\ \text{linear} \,\Big|\\ e^{tA}:\rset[x_1,\dots,x_n]\to\rset[x_1,\dots,x_n]\
\text{for all}\ t\in [0,\varepsilon)\}.
\end{multline*}
\end{dfn}

Clearly,
\[\fg = \rset\cdot\fg,\]
i.e., $\fg$ is a cone.
The following theorem gives characterizations of $\fg$, also showing that $\fg$ is independent on $\varepsilon>0$ and hence that it is a cone.
Before we state the theorem, we need the following lemma.

\begin{lem}\label{lem:supDegreeBound}
Let $n\in\nset$ and let
\[c: [0,1]\to\rset[x_1,\dots,x_n]\]
be continuous.
Then
\[\sup_{t\in [0,1]} \deg c(t)\ <\ \infty.\]
\end{lem}
\begin{proof}
Since $c$ is continuous, for every $s\in\cset^{\nset_0^n}$,
\[\langle s,c(t)\rangle = L_s(c(t))\]
is continuous.
Assume
\[\sup_{t\in [0,1]} \deg c(t) = \infty.\]
By compactness of $[0,1]$ and by choosing a convergent subsequence, there exists a sequence $(t_k)_{k\in\nset}$ in $[0,1]$ and a $t_0\in [0,1]$ such that
\[\lim_{k\to\infty} t_k = t_0, \qquad \deg c(t_k) < \deg c(t_{k+1}), \qquad\text{and}\qquad \lim_{k\to\infty} \deg c(t_k) = \infty.\]
For $k=1$, choose $\tilde{s}_\alpha\in\cset$ with $\alpha\in\nset_0^n$ and $|\alpha|\leq \deg c(t_1)$ such that
\[\sum_{\alpha\in\nset_0^n: |\alpha|\leq \deg c(t_1)} \tilde{s}_\alpha\cdot c_\alpha(t_1) = 1.\]
Since $\deg c(t_k) < \deg c(t_{k+1})$ for all $k\in\nset$, by proceeding in this manner we can choose $\tilde{s}_\alpha\in\cset$ for all $\alpha\in\nset_0^n$ with $\deg c(t_k) < |\alpha| \leq \deg c(t_{k+1})$ and
\[\sum_{\alpha\in\nset_0^n: |\alpha|\leq \deg c(t_{k+1})} \tilde{s}_\alpha\cdot c_\alpha(t_{k+1}) = \sum_{\alpha\in\nset_0^n} \tilde{s}_\alpha\cdot c_\alpha(t_{k+1}) = k+1,\]
i.e., by induction, there exists a sequence $\tilde{s} = (\tilde{s}_\alpha)_{\alpha\in\nset_0^n}\in\cset^{\nset_0^n}$ such that
\[\langle \tilde{s}, c(t_k)\rangle = \sum_{\alpha\in\nset_0^n} \tilde{s}_\alpha\cdot c_\alpha(t_k) = k\]
for all $k\in\nset$.
Hence,
\[\lim_{k\to\infty} \langle\tilde{s},c(t_k)\rangle = \infty.\]
This contradicts continuity of $c$ on $[0,1]$ and hence we proved the assertion.
\end{proof}

We can now characterize $\fg$.

\begin{thm}\label{thm:gCharac}
Let $n\in\nset$ and let
\[A:\rset[x_1,\dots,x_n]\to\rset[x_1,\dots,x_n]\]
be linear.
Then the following are equivalent:
\begin{enumerate}[(i)]
\item $A\in\fg$.

\item $\displaystyle \sup_{k\in\nset_0} \deg A^k x^\alpha < \infty$ for all $\alpha\in\nset_0^n$.

\item $\displaystyle \sup_{k\in\nset_0} \deg A^k p < \infty$ for all $p\in\rset[x_1,\dots,x_n]$.

\item For all $i\in\nset_0$, there exist subspaces $V_i\subseteq\rset[x_1,\dots,x_n]$ with
\begin{enumerate}[(a)]
\item $\dim V_i < \infty$,

\item $\displaystyle \bigcup_{i\in\nset_0} V_i = \rset[x_1,\dots,x_n]$, and

\item $AV_i \subseteq V_i$ for all $i\in\nset_0$.
\end{enumerate}

\end{enumerate}
\end{thm}
\begin{proof}
The implications ``(iii) $\Leftrightarrow$ (ii) $\Rightarrow$ (i)'' and ``(iv) $\Rightarrow$ (i)'' are clear.
The implication ``(ii) $\Leftrightarrow$ (iii)'' follows from linearity of $A$.
It remains to show the remaining implications.

(i) $\Rightarrow$ (ii):
Since
\[e^{tA}:\rset[x_1,\dots,x_n]\to\rset[x_1,\dots,x_n],\]
is continuous for each $\alpha\in\nset_0^n$, by \Cref{lem:supDegreeBound}, there exists a $D=D(\alpha)\in\nset_0^n$ such that
\[e^{tA}x^\alpha\ \subseteq\ \rset[x_1,\dots,x_n]_{\leq D}\]
for all $t\in [0,\varepsilon/2]$.
Hence,
\[A^k x^\alpha\ =\ \partial^k_t e^{tA}x^\alpha \Big|_{t=0}\ \in\ \rset[x_1,\dots,x_n]_{\leq D}\]
for all $k\in\nset_0$, which proves (ii).
Here, $\partial_t$ only needs to be the right-sided derivative.

[(i) $\Leftrightarrow$ (ii) $\Leftrightarrow$ (iii)] $\Rightarrow$ (iv):
Let $\alpha\in\nset_0^n$.
By (i) and (ii) (\Cref{lem:supDegreeBound}), there exists a $D=D(\alpha)\in\nset_0$ with
\[e^{tA}x^\alpha,\ A^k x^\alpha\ \in\ \rset[x_1,\dots,x_n]_{\leq D}\]
for all $k\in\nset_0$ and all $t\in [0,\varepsilon/2]$.
Set
\[V_{\alpha,0} := \rset\cdot x^\alpha.\]
By (ii),
\[AV_{\alpha,0}\quad \subseteq\quad \rset[x_1,\dots,x_n]_{\leq D}.\]
For all $i\in\nset_0$, define
\[V_{\alpha,i+1} := V_{\alpha,i} + AV_{\alpha,i},\]
i.e.,
\[V_{\alpha,k}=\rset\cdot x^\alpha +\rset\cdot Ax^\alpha +\dots+\rset\cdot A^k x^\alpha\]
for all $k\in\nset_0$.
By (ii) and the definition of $V_{\alpha,i+1}$,
\[V_{\alpha,i} \quad\subseteq\quad V_{\alpha,i+1} \quad\subseteq\quad \rset[x_1,\dots,x_n]_{\leq D}\]
and hence
\begin{equation}\label{eq:increasingViSequence}
V_{\alpha,0} \quad\subseteq\quad V_{\alpha,1} \quad\subseteq\quad \dots \quad\subseteq\quad \rset[x_1,\dots,x_n]_{\leq D}.
\end{equation}
Since (\ref{eq:increasingViSequence}) is an increasing sequence of finite dimensional vector spaces $V_{\alpha,i}$ bounded from above by the finite dimensional vector space
\[\rset[x_1,\dots,x_n]_{\leq D},\]
there exists an index $I(\alpha)\in\nset_0$ such that
\[V_{\alpha,I(\alpha)} = V_{\alpha,I(\alpha)+1},\]
i.e.,
\[\dim V_{\alpha,I(\alpha)} < \infty \qquad\text{and}\qquad AV_{\alpha,I(\alpha)} \subseteq V_{\alpha,I(\alpha)}.\]
Since $\alpha\in\nset_0^n$ was arbitrary and $x^\alpha\in V_{\alpha,I(\alpha)}$,
\[\rset[x_1,\dots,x_n] = \bigcup_{\alpha\in\nset_0^n} V_{\alpha,I(\alpha)}.\]
Since $\nset_0^n$ is countable, we proved (iv) (a) -- (c).
\end{proof}

\Cref{thm:gCharac} (iv) tells us that
\[e^{tA}:\rset[x_1,\dots,x_n]\to\rset[x_1,\dots,x_n]\]
is only the matrix exponential function
\[e^{tA}|_{V_i},\]
since for an operator $A\in\fg$ we only need to know $A$ on the finite dimensional invariant subspace $V_i$, i.e., its calculation is easy as soon as the $V_i$ for $A$ are known.
An algorithm for the calculation of these $V_i$ is explicitly given in the proof of \Cref{thm:gCharac}.
Therefore, by \Cref{thm:gCharac} (vi), we have the following corollary.

\begin{cor}\label{cor:taylor}
Let $n\in\nset$ and $A\in\fg$.
Then
\[e^{tA}\; =\; \sum_{k\in\nset_0} \frac{t^k\cdot A^k}{k!}\; =\; \lim_{k\to\infty} \left(1 + \frac{t\cdot A}{k} \right)^k\; =\; \lim_{k\to\infty} \left(1 - \frac{t\cdot A}{k} \right)^{-k}\]
for all $t\in\rset$.
\end{cor}
\begin{proof}
Since $A\in\fg$, by \Cref{thm:gCharac} (iv), there exists a family
\[\cV = \{V_i\}_{i\in\nset_0}\]
of subspaces $V_i\subseteq\rset[x_1,\dots,x_n]$ with
\[\rset[x_1,\dots,x_n] = \bigcup_{i\in\nset_0} V_i, \qquad AV_i\subseteq V_i, \qquad\text{and}\qquad \dim V_i < \infty\]
for all $i\in\nset_0$.
Since all $V_i$ are finite dimensional and invariant under $A$, $A|_{V_i}$ is a linear operator on a finite dimensional vector space (matrix) and for matrices all three equalities hold.
\end{proof}

\begin{rem}
\Cref{cor:taylor} emphasizes the special situation we have on the LF-space $\rset[x_1,\dots,x_n]$ compared to a Hilbert space like $L^2(\rset,\rset)$.
On $L^2(\rset,\rset)$ the translation group
\[(T_t)_{t\in\rset} \quad\text{with}\quad (T_tf)(x) := f(x+t)\quad\text{and}\quad f\in L^2(\rset,\rset)\]
is generated by $\partial_x$.
But no operator $T_t$ with $t\neq 0$ is given by its (formal) Taylor expansion
\[e^{t\cdot\partial_x} = \sum_{k\in\nset_0} \frac{t^k\cdot \partial_x^k}{k!}\]
on all $L^2(\rset,\rset)$, since $f\in L^2(\rset,\rset)$ even does not need to be differentiable.

Also the heat equation, i.e., convolution with the heat kernel, is solved by the heat kernel generated by
\[\Delta := \partial_x^2.\]
On $L^2(\rset,\rset)$ this generates only a \emph{semi}-group ($t\geq 0$), i.e., no time reversal is possible in general.
By \Cref{thm:gCharac}, on $\rset[x_1,\dots,x_n]$ any semi-group
\[(e^{tA})_{t\geq 0}\]
extends uniquely to a group
\[(e^{tA})_{t\in\rset}\]
with the same generator $A$, i.e., in \Cref{dfn:fg} it was sufficient to have
\[e^{tA}:\rset[x_1,\dots,x_n]\to\rset[x_1,\dots,x_n]\]
for all $t\in [0,\varepsilon)$ for some $\varepsilon>0$ and we immediately get it for all $t\in\rset$.
\exmsymbol
\end{rem}

\section{Regular Fr\'echet Lie Algebras in $\fg$}
\label{sec:fgAlgebras}

We know want to investigate regular Fr\'echet Lie algebras in $\fg$.

We have seen that $\fg$ is a cone.
However, $\fg$ is not convex and it is especially not a vector space, i.e., there are $A,B\in\fg$ such that $A+B\notin\fg$.
We have the following example of such maps $A$ and $B$.

\begin{exm}\label{exm:ABshiftEvenOdd}
Let $n=1$.
Define the linear operators
\[A:\rset[x]\to\rset[x] \qquad\text{and}\qquad B:\rset[x]\to\rset[x]\]
by
\[Ax^k := \begin{cases}
x^k & \text{for}\ k=2m\ \text{with}\ m\in\nset_0,\\
x^{k+1} & \text{for}\ k=2m-1\ \text{with}\ m\in\nset
\end{cases}\]
and
\[Bx^k := \begin{cases}
x^{k+1} & \text{for}\ k=2m\ \text{with}\ m\in\nset_0,\\
x^k & \text{for}\ k=2m-1\ \text{with}\ m\in\nset
\end{cases}\]
for all $k\in\nset_0$ with linear extension to all $\rset[x]$.
Then
\[A\rset[x]_{\leq 2m}\ \subseteq\ \rset[x]_{\leq 2m} \qquad\text{and}\qquad B\rset[x]_{\leq 2m+1}\ \subseteq\ \rset[x]_{\leq 2m+1}\]
for all $m\in\nset_0$.
Since $A|_{\rset[x]_{\leq 2m}}$ and $B|_{\rset[x]_{\leq 2m+1}}$ are linear operators on finite dimensional spaces, i.e., matrices, $e^{tA}$ and $e^{tB}$ are well-defined for every $p\in\rset[x]$ and hence they are well-defined as linear maps $\rset[x]\to\rset[x]$: $A,B\in\fg$.
However,
\[(A+B)x^k\quad =\quad \begin{cases}
x^k + x^{k+1} & \text{for}\ k=2m\ \text{with}\ m\in\nset_0,\\
x^{k+1} + x^k & \text{for}\ k=2m-1\ \text{with}\ m\in\nset
\end{cases}\quad =\quad x^k + x^{k+1}\]
for all $k\in\nset_0$ which violates \Cref{thm:gCharac} (ii), i.e.,
\[\deg \left(\sum_{i=0}^d \frac{t^i}{i!}\cdot (A+B)^i 1\right) = d \quad\xrightarrow{d\to\infty}\quad\infty\]
for all $t\neq 0$.
Hence, $e^{t(A+B)}$ is not a map $\rset[x]\to\rset[x]$ and $A+B\notin\fg$.
\exmsymbol
\end{exm}

Hence, $\fg$ is a structure that is already larger than a (regular Fr\'echet) Lie algebra.

The previous example showed that if for a linear operator
\[A:\rset[x_1,\dots,x_n]\to\rset[x_1,\dots,x_n]\]
there exist $d_0 < d_1 < d_2 < \dots$ in $\nset_0$ with
\begin{equation}\label{eq:invariantDegree}
A\rset[x_1,\dots,x_n]_{\leq d_i}\ \subseteq\ \rset[x_1,\dots,x_n]_{\leq d_i}
\end{equation}
for all $i\in\nset_0$, then $A\in\fg$.
This was an essential property used in \cite{cuchie12}, \cite{didio24posPresConst}, and \cite{didio25KPosPresGen} to calculate and work with $e^{tA}$.
All linear operators $A$ with (\ref{eq:invariantDegree}) form a (regular Fr\'echet) Lie algebra, see \Cref{cor:LieAlgInvSubSpaces}.
But before we prove that, let us have a look at some examples to get more familiar with \Cref{thm:gCharac} (iv).
The next example shows that (\ref{eq:invariantDegree}) is sufficient for $A\in\fg$, but (\ref{eq:invariantDegree}) is not necessary for $A\in\fg$.

\begin{exm}
Let $n=1$ and let
\[(p_i)_{i\in\nset}= (2,3,5,7,11,\dots)\]
be the list of prime numbers $p_i$.
Define the linear map
\[A:\rset[x]\to\rset[x]\]
by
\[Ax^k := \begin{cases}
x^{p_{m+1}} & \text{for}\ k=2m\ \text{with}\ m\in\nset,\\
0 & \text{else}.
\end{cases}\]
Then
\[A\rset[x]_{\leq d}\ \not\subseteq\ \rset[x]_{\leq d}\]
for all $d\in\nset$ with $d\geq 8$.
However,
\[A1 = 0, \quad Ax = 0, \quad Ax^2 = x^{p_2} = x^3,\quad Ax^3 = 0,\quad Ax^4 = x^{p_3} = x^5,\quad \dots\]
and therefore
\[A^2 x^k = 0\]
for all $k\in\nset_0$, i.e.,
\[e^{tA}:\rset[x]\to\rset[x]\]
for all $t\in\rset$ and $A\in\fg$.
\exmsymbol
\end{exm}

In general, in \Cref{thm:gCharac} (vi) we can only have
\[V_i \subseteq V_{i+1}\]
for the invariant subspaces $V_i$ of an operator $A\in\fg$, as the next example shows.

\begin{exm}\label{exm:jordan}
Let $n=1$ and
\[A = \partial_x\in\fg.\]
Then every finite dimensional invariant subspace
\[V_i\subseteq\rset[x]\]
of $\partial_x$ is of the form $\rset[x]_{\leq d}$ for some $d\in\nset_0$.
To see this, let $V$ be an invariant subspace of $\partial_x$ with $\dim V < \infty$.
Let $p\in V$ be of highest degree, i.e.,
\[\deg q \leq \deg p =: d\]
for all $q\in V$ and hence
\[V\subseteq\rset[x]_{\leq d}.\]
Then
\[\partial_x^k p\in V \qquad\text{and}\qquad \deg (\partial_x^k p) = d - k\]
for all $k=1,\dots,d$.
Since $\deg (\partial_x^k p) = d-k$, the set
\[\left\{\partial_x^k p\right\}_{k=0}^d\]
is linearly independent and it spans $\rset[x]_{\leq d}$, i.e.,
\[\rset[x]_{\leq d}\subseteq V.\]
In summary,
\[V=\rset[x]_{\leq d}\]
for some $d\in\nset_0$.
\exmsymbol
\end{exm}

The operator $\partial_x$ is a Jordan block of infinite size.
Therefore, we can in general not hope for
\[\rset[x_1,\dots,x_n]\]
to be split into disjoint invariant subspaces $V_i$:
\begin{equation}\label{eq:decompo}
\rset[x_1,\dots,x_n] = V_1 \oplus V_2 \oplus \dots.
\end{equation}
Even when we allow
\[\dim V_i = \infty\]
in the decomposition (\ref{eq:decompo}), then \Cref{exm:jordan} shows that for $\partial_x$ only the trivial decomposition
\[\rset[x] = V_1\]
exists.

Additionally, for the invariant subspaces $V_i$ of $A\in\fg$, we have the following.

\begin{cor}
Let $n\in\nset$, let $A\in\fg$, and let
\[V\subseteq\rset[x_1,\dots,x_n]\]
be a finite or infinite dimensional subspace.
Then the following are equivalent:
\begin{enumerate}[(i)]
\item $AV\subseteq V$.

\item $e^{tA}V \subseteq V$ for all $t\in\rset$.
\end{enumerate}
\end{cor}
\begin{proof}
The direction ``(i) $\Rightarrow$ (ii)'' follows from
\[e^{tA} = \sum_{k\in\nset_0} \frac{1}{k!}\cdot t^k\cdot A^k\]
by \Cref{cor:taylor}.
The direction ``(ii) $\Rightarrow$ (i)'' follows from
\[A = \partial_t e^{tA}\big|_{t=0}.\qedhere\]
\end{proof}

We have seen in \Cref{exm:ABshiftEvenOdd} that $\fg$ is not closed under addition.
The following example shows that $\fg$ is also not closed under multiplication.

\begin{exm}[\Cref{exm:ABshiftEvenOdd} continued]\label{exm:ABshiftEvenOdd2}
Let
\[A,B:\rset[x]\to\rset[x]\]
be given as in \Cref{exm:ABshiftEvenOdd}.
Then
\begin{equation}\label{eq:AB}
AB x^{2m} = x^{2m+2},\qquad ABx^{2m+1} = x^{2m+2},
\end{equation}
and
\begin{equation}\label{eq:BA}
BAx^{2m} = x^{2m+1},\qquad BA x^{2m+1} = x^{2m+3},
\end{equation}
for all $m\in\nset_0$, i.e., by \Cref{thm:gCharac} (ii), $AB\notin\fg$ as well as $BA\notin\fg$.
\exmsymbol
\end{exm}

If two operators $A,B\in\fg$ possess a common family
\[\{V_i\}_{i\in\nset_0}\]
of invariant and finite dimensional subspace $V_i$ as provided by \Cref{thm:gCharac} (iv), then
\[A+B\in\fg, \quad AB\in\fg,\quad\text{and}\quad BA\in\fg\]
as well as the Lie bracket $[\,\cdot\,,\,\cdot\,]$ is given by
\[[A,B] := AB - BA\ \in\fg.\]

In general, $\fg$ is not closed under this Lie bracket.
The operators $A,B\in\fg$ in \Cref{exm:ABshiftEvenOdd} provide an example of $[A,B]\notin\fg$.

\begin{exm}[Examples \ref{exm:ABshiftEvenOdd} and \ref{exm:ABshiftEvenOdd2} continued]\label{exm:ABshiftEvenOdd3}
Let
\[A:\rset[x]\to\rset[x] \qquad\text{and}\qquad B:\rset[x]\to\rset[x]\]
be defined as in \Cref{exm:ABshiftEvenOdd}.
By (\ref{eq:AB}) and (\ref{eq:BA}) in \Cref{exm:ABshiftEvenOdd2},
\[[A,B]x^{2m} = x^{2m+2} - x^{2m+1} \qquad\text{and}\qquad [A,B]x^{2m+1} = x^{2m+2} - x^{2m+3}\]
for all $m\in\nset_0$, i.e., $[A,B]\notin\fg$ by \Cref{thm:gCharac} (ii).
\exmsymbol
\end{exm}

In summary, for $A,B\in\fg$, in order to have
\[A+B\in\fg,\quad AB\in\fg,\quad BA\in\fg, \quad\text{and}\quad [A,B]\in\fg\]
it is sufficient that the operators $A$ and $B$ possess a common family
\[\cV=\{V_i\}_{i\in\nset_0}\]
of invariant finite dimensional subspaces $V_i$ in \Cref{thm:gCharac} (iv).

\begin{cor}\label{cor:LieAlgInvSubSpaces}
Let $n\in\nset$ and let
\[\cV := \{V_i\}_{i\in\nset_0}\]
be a family of subspaces $V_i\subseteq\rset[x_1,\dots,x_n]$ such that
\[\rset[x_1,\dots,x_n]=\bigcup_{i\in\nset_0} V_i \qquad\text{and}\qquad \dim V_i < \infty\]
for all $i\in\nset_0$.
Then
\[\fg_\cV := \big\{A\in\fg \,\big|\, AV_i\subseteq V_i\ \text{for all}\ i\in\nset_0\big\}\]
is a regular Fr\'echet Lie algebra in $\fg$ with Lie bracket
\[[\,\cdot\,,\,\cdot\,]:\fg_\cV\times\fg_\cV\to\fg_\cV,\quad (A,B)\mapsto [A,B] := AB - BA.\]
The set $G_\cV := \exp(\fg_\cV)$ is the corresponding regular Fr\'echet Lie group.
\end{cor}

The question is, are all regular Fr\'echet Lie algebras in $\fg$ contained in some $\fg_\cV$?
For finite dimensional Lie algebras in $\fg$ this is true.

\begin{cor}\label{cor:subspaceLieAlgebras}
Let $n\in\nset$ and let $\tilde{\fg}$ be a finite dimensional regular Fr\'echet Lie algebra in $\fg$.
Then there exists a family
\[\cV = \{V_i\}_{i\in\nset_0}\]
of finite dimensional subspaces $V_i$ of $\rset[x_1,\dots,x_n]$ with
\[\rset[x_1,\dots,x_n] = \bigcup_{i\in\nset_0} V_i\]
such that $\tilde{\fg}\subseteq\fg_\cV$.
\end{cor}
\begin{proof}
Let $g_1,\dots,g_N$ be a vector space basis of $\tilde{\fg}$, i.e.,
\[\dim \tilde{\fg}=:N\in\nset.\]
Let
\begin{equation}\label{eq:a1aN}
a_1,\dots,a_N\in\rset \quad\text{with}\quad a_1^2 + \dots + a_N^2 = 1.
\end{equation}
Then for any $\alpha\in\nset_0^n$ there exists a $D = D(\alpha)\in\nset$ such that
\[(a_1 g_1 + \dots + a_N g_N)^k x^\alpha \quad\subseteq\quad \rset[x_1,\dots,x_n]_{\leq D}\]
for all $a_1,\dots,a_N$ with (\ref{eq:a1aN}), since
\[(a_1 g_1 + \dots + a_N g_N)^k\]
is  continuous in $a_1,\dots,a_N$ and the set (\ref{eq:a1aN}) is compact.
From here, continue as in the proof of \Cref{thm:gCharac} step ``[(i) $\Leftrightarrow$ (ii)] $\Rightarrow$ (iv)''.
\end{proof}

The reader will of course immediately see that in the proof of the previous corollary we did not need that the finite dimensional Lie algebra $\tilde{\fg}$ is closed with respect to the Lie bracket $[\,\cdot\,,\,\cdot\,]$.
Only that $\tilde{\fg}$ is a finite dimensional vector space was used.
Hence, we have the following immediate consequence.

\begin{cor}
Let $n\in\nset$ and $A,B\in\fg$.
Then the following are equivalent:
\begin{enumerate}[(i)]
\item $\alpha A + \beta B\in\fg$ for all $\alpha,\beta\in\rset$.

\item There exists a family
\[\cV = \{V_i\}_{i\in\nset_0}\]
of subspaces
$V_i\subseteq\rset[x_1,\dots,x_n]$
such that
\[\rset[x_1,\dots,x_n] = \bigcup_{i\in\nset_0} V_i \qquad\text{and}\qquad \dim V_i <\infty\]
for all $i\in\nset_0$ with $A,B\in\fg_\cV$.
\end{enumerate}
\end{cor}

In summary, we described all finite dimensional Lie algebras in $\fg$, i.e., they are always contained in some $\fg_\cV$.
$\fg_\cV$ is infinite dimensional.
But an infinite dimensional regular Fr\'echet Lie algebra of $\fg$ does not need to be contained in some $\fg_\cV$.
The following is an example of such an infinite dimensional Lie algebra.

\begin{exm}\label{exm:nonSubspaceLieAlgebra}
Let $n\in\nset$, let $y\in\rset^n$, and let
\[l_y: \rset[x_1,\dots,x_n]\to\rset,\qquad f\mapsto l_y(f) := f(y)\]
be the point evaluation at $y$.
Define
\[\fg_y := \big\{l_y\cdot p \,\big|\, p\in\rset[x_1,\dots,x_n]\big\}.\]
Then
\begin{enumerate}[(i)]
\item $\fg_y\subseteq\fg$,
\item $\alpha A+\beta B\in\fg_y$ for all $A,B\in\fg_y$ and $\alpha,\beta\in\rset$,
\item $AB\in\fg_y$ for all $A,B\in\fg_y$,
\item $[A,B]\in\fg_y$ for all $A,B\in\fg_y$, and
\item $\fg_y$ is closed,
\end{enumerate}
i.e., $\fg_y$ is a regular Fr\'echet Lie algebra in $\fg$, since
\[\alpha A+\beta B=l_y\cdot(\alpha p+\beta q)\qquad\text{and}\qquad AB=q(y)\cdot l_y\cdot p\]
with $Af = f(y)\cdot p$, $Bf = f(y)\cdot q$, and $\alpha,\beta\in\rset$.
To see that $\fg_y$ is closed let
\[A_k = l_y\cdot p_k \qquad\text{with}\qquad A_k\xrightarrow{k\to\infty} \big[A:\rset[x_1,\dots,x_n]\to\rset[x_1,\dots,x_n]\big].\]
Hence,
\[A_k1 = p_k\to p\in\rset[x_1,\dots,x_n]\]
as $k\to\infty$ and $A=l_y\cdot p\in\fg_y$.
But since $\deg p$ in $A = l_y\cdot p$ is not bounded, $\fg_y$ is not contained in any $\fg_\cV$.
\exmsymbol
\end{exm}

In \Cref{dfn:frechetLieGroup} we have seen what a regular Fr\'echet Lie group and its Lie algebra are.
The Examples \ref{exm:ABshiftEvenOdd}, \ref{exm:ABshiftEvenOdd2}, and \ref{exm:ABshiftEvenOdd3} show that
$\fg$
is not a (regular Fr\'echet) Lie algebra of any regular Fr\'echet Lie group $G$.
It is too large and it is not even a vector space.

However, since $\fg$ contains all operators $A$ such that
\[e^{tA}:\rset[x_1,\dots,x_n]\to\rset[x_1,\dots,x_n]\]
for all $t\in\rset$, $\fg$ contains all regular Fr\'echet Lie algebras in our case.
Hence, $\fg$ is a simple toy example to further study regular Fr\'echet Lie groups and their algebras.

\section{Three-Line-Proof of Stochel's Theorem}
\label{sec:stochel}

Since we are interested in $K$-positivity preservering semi-groups and these are connected to $K$-moment sequences by \Cref{thm:KposPres}, we deal in this section with $K$-moment sequences.

Recall from \Cref{dfn:riesz}, for a real sequence $s = (s_\alpha)_{\alpha\in\nset_0^n}$ and a polynomial
\[p(x) = \sum_{\alpha\in\nset_0^n} p_\alpha\cdot x^\alpha \quad\in\rset[x_1,\dots,x_n],\]
the \emph{Riesz functional}
\[L_s:\rset[x_1,\dots,x_n]\to\rset\]
is defined by
\[L_s(p) = \sum_{\alpha\in\nset_0^n} p_\alpha\cdot s_\alpha.\]
We get the following three-line-proof of Stochel's Theorem \cite{stochel01}.

\begin{thm}\label{thm:stochel}
Let $n\in\nset$ and $K\subseteq\rset^n$ closed.
The following are equivalent:
\begin{enumerate}[(i)]
\item $s$ is a $K$-moment sequence.

\item For all $d\in\nset_0$,
\[s|_{\leq d} := (s_\alpha)_{\alpha\in\nset_0^n:|\alpha|\leq d}\]
is a truncated $K$-moment sequence.
\end{enumerate}
\end{thm}
\begin{proof}
(i) $\Rightarrow$ (ii): Clear.

(ii) $\Rightarrow$ (i): For $p\in\pos(K)$, $L_s(p) \overset{\deg p \leq d}{=} L_{s|_{\leq d}}(p) \overset{\text{(ii)}}{\geq} 0$ and, by Haviland's Theorem \cite{havila36}, $s$ is a $K$-moment sequence.
\end{proof}

Note, in (ii) actually a weaker property is used:
\begin{enumerate}[(ii')]
\item For all $d\in\nset_0$ and $p\in\pos(K)_{\leq d}$, $L_{s|_\leq d}(p)\geq 0$.
\end{enumerate}
The truncated sequences $s|_{\leq d}$ do not need to be truncated $K$-moment sequences.

\begin{rem}
This three-line-proof was remarked by the author at the LAW25 conference.
Jan Stochel was present in the lecture hall when the author made this remark.
In the afternoon after this remark, Jan Stochel and Zenon Jablonski took the author aside and sat down on a bench in front of the Fakulteta za pomorstvo in promet Univerze v Ljubljani where the conference took place.
The author, surrounded by Jan Stochel sitting on the right and Zenon Jablonski on the left, wrote down this three-line-proof.
The original notes are shown in \Cref{fig:threelineproof}.
\begin{figure}[tbh!]\centering
\includegraphics[width=0.75\textwidth]{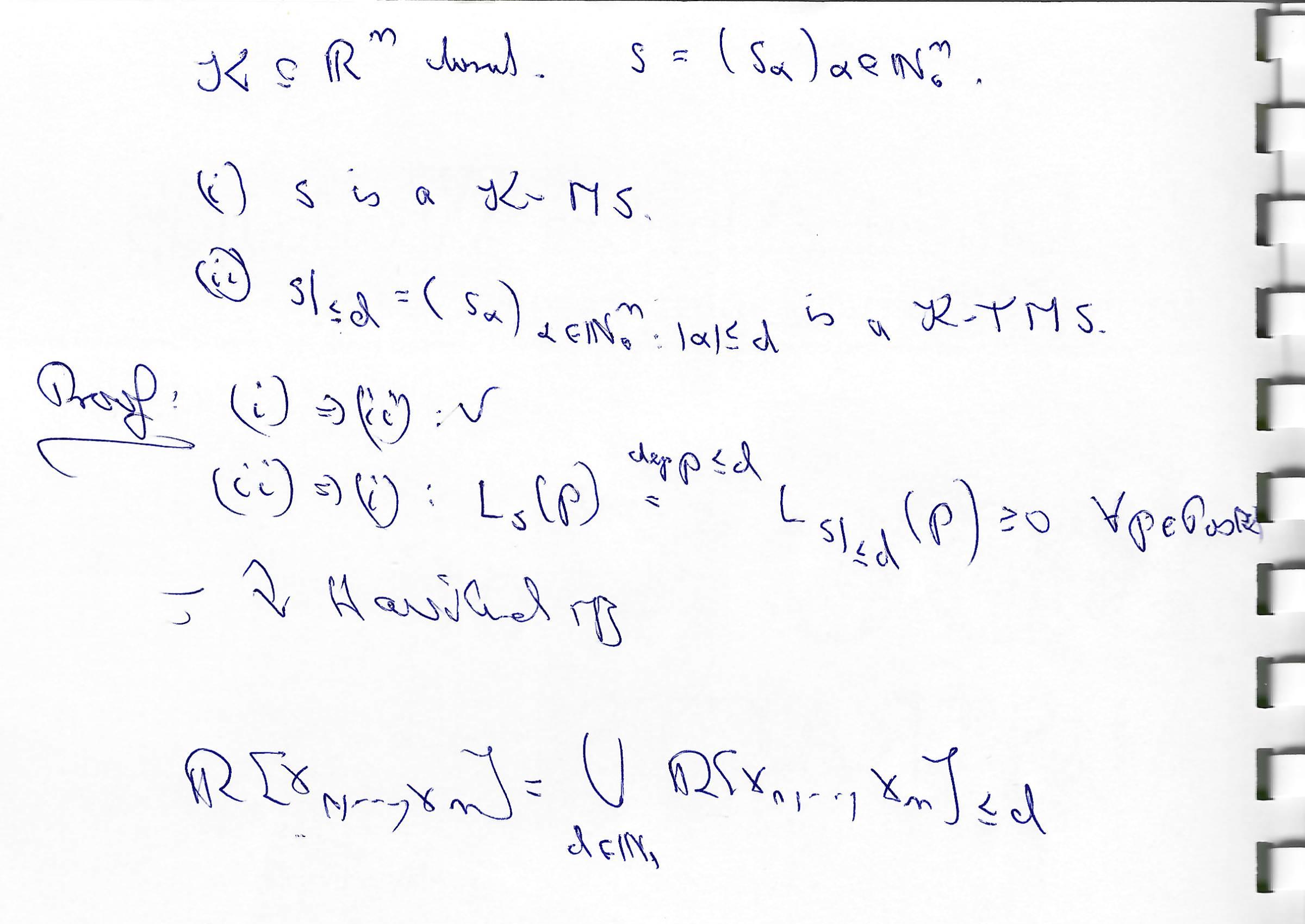}
\caption{The original handwritten notes by the author of the three-line-proof of Stochel's Theorem \cite{stochel01}, presented to Jan Stochel and Zenon Jablonski one afternoon at the LAW25 conference in front of the Fakulteta za pomorstvo in promet Univerze v Ljubljani.\label{fig:threelineproof}}
\end{figure}
While the author pointed out, that this three-line-proof must have been known before and that he was possibly not the first to notice that, Jan Stochel and Zenon Jablonski where not aware of this three-line-proof but immediately saw, that the weaker condition (ii') is enough and that the theorem can be extended to graded algebras as long as an equivalent version of Haviland's Theorem is known there, since the proof is based on the simple observations
\[L_s(p) \overset{\deg p \leq d}{=} L_{s|{_\leq d}}(p)\]
and
\[\rset[x_1,\dots,x_n] = \bigcup_{d\in\nset_0} \rset[x_1,\dots,x_n]_{\leq d}.\]
Both encouraged the author to publish this proof.
Hence, it is included here.
The original proof \cite{stochel01} also extends to $*$-semi-groups, where an analogue of Haviland's Theorem is available.

Personally, I want to express my gratidute to both, Jan Stochel and Zenon Jablonski, for their kindness, openness, and impartiality to listen to me and to let me show this proof to them.
While I was advised several times ``not to fuck with the big guys'' since they can hurt my career immensely, I am very happy to be able to meet Jan Stochel and Zenon Jablonski.
It is refreshing that there are (still) mathematicians out there, who value good arguments and efficient proofs more than personal approval.
I am very happy to personally meet them.
\exmsymbol
\end{rem}

\section{Generators of $K$-Positivity Preserving Semi-Groups}
\label{sec:KposGen}

We characterized $\fg$ (\Cref{thm:gCharac}) and how to work with generators $A\in\fg$ (\Cref{cor:taylor}).
We have also seen that in applications Markov processes come from generators $A\in\fg$ with (\ref{eq:invariantDegree}) and that the semi-groups
\[\left( e^{tA}\right)_{t\geq 0}\]
are $K$-positivity preseving \cite{cuchie12}.
Hence, we continue our investigations from \cite{didio24posPresConst} and \cite{didio25KPosPresGen} on generators of $K$-positivity preserving semi-groups.
At first, let us have a well-known example.
It is a group, not just a semi-group.

\begin{exm}
Let
\[A := x\partial_y - y\partial_x.\]
Then $A$ generates the rotation group on the polynomials $\rset[x,y]$ and is therefore a generator of a $K$-positivity preservering group $(e^{tA})_{t\in\rset}$ for any ball
\[K := B_r(0)\]
centered at $0\in\rset^2$ with radius $r\geq 0$.
\exmsymbol
\end{exm}
\begin{proof}
That $y\partial_x - x\partial_y$ is the generator of the rotation is the same calculation as on $L^2(\rset^2,\rset)$.
\end{proof}

The second example is non-trivial and it was open for several years.
We solve it for the first time and we need the techniques from the previous sections.

\begin{exm}\label{exm:heat-11}
Let
\[A := (1-x^2)\cdot\partial_x^2.\]
Then
\[e^{tA}\pos([-1,1,]) \;\subseteq\; \pos([-1,1])\]
for all $t\geq 0$, i.e.,
\[(1-x^2)\cdot\partial_x^2\in\fg\]
is a generator of a $[-1,1]$-positivity preserving semi-group.
\exmsymbol
\end{exm}
\begin{proof}
Let $d\in\nset_0$.
Then
\[A\rset[x]_{\leq d}\subseteq\rset[x]_{\leq d},\]
i.e., $A\in\fg$.
Additionally,
\[g: [0,\infty)\times [-1,1]\times [-1,1]\times \pos([-1,1])\to\rset\]
defined by
\[g(t,x,y,p) := \left( e^{tA_y} p\right)(x)\]
is continuous in $(t,x,y,p)$.

Let $\cB$ be a compact base of $\pos([-1,1])_{\leq d}$.
For every $y\in [-1,1]$ and $p\in\cB$, there exists a maximal $\tau_{y,p}\in (0,\infty]$ such that
\[\left( e^{tA_y}p \right)(y) \geq 0\]
for all $t\in [0,\tau_{y,p}]$.
To see this, we have the following three cases:
\begin{enumerate}[\quad (a)]
\item $y\in (-1,1)$ and $p(y)>0$: Clear by continuity of $p$ and the Gaussian kernel convolution $e^{tA_y}$.

\item $y\in (-1,1)$ and $p(y) = 0$: Since $p$ has only isolated zeros, $p(x)>0$ in a neighborhood of $y$ with $x\neq y$, continuity of $p$ and the Gaussian kernel convolution give the statement.

\item $y = \pm 1$: $A_{\pm 1} = 0$ and hence
\[(e^{tA_y} p)(y) = p(y) = (e^{tA} p)(y).\]
The last equality follows from the Taylor expansion (\Cref{cor:taylor}), since
\[A^k = (1-x^2)\cdot\partial_x^2 \cdot{\ldots}\]
for all $k\in\nset$ and hence all $A^k$ with $k\geq 1$ vanish at $x=y=\pm 1$.
\end{enumerate}
Hence,
\[\tau_d := \inf_{y\in [-1,1], p\in\cB} \tau_{y,p} > 0\]
and therefore
\[e^{tA_y}\pos([-1,1])_{\leq d} \subseteq \pos([-1,1])_{\leq d}\]
for all $t\in [0,\tau_d]$ with $\tau_d>0$ and $y\in [-1,1]$.

For a function $f$ in the variables $x$ and $y$ in the same domain, define
\[(\Ev_{y=x} f)(x) := f(x,x).\]
Hence, for $p\in\pos([-1,1])_{\leq d}$ and $t\in [0,\infty)$, since
\[\frac{t}{k}\in (0,\tau_d]\]
for $k\gg 1$,
\begin{align*}
e^{tA}p &= \lim_{k\to\infty} \left( 1 + \frac{tA}{k} \right)^k p\\
&= \lim_{k\to\infty} \left( \Ev_{y=x}\circ \left( 1 + \frac{tA_y}{k} \right)\right)^k p\\
&= \lim_{k\to\infty} \left( \Ev_{y=x}\circ e^{t A_y/k} \right)^k p \qquad\quad\in \pos([-1,1])_{\leq d}.
\end{align*}
Since $p\in\pos([-1,1])_{\leq d}$ and $d\in\nset_0$ were arbitrary, the statement is proved.
\end{proof}

We have the following resolvent description of all generators $A\in\fg$ of $K$-positivity preserving semi-groups for any closed $K\subseteq\rset^n$.

\begin{thm}
Let $n\in\nset$, let $K\subseteq\rset^n$ be closed, and let $A\in\fg$ with the family
\[\cV = \{V_i\}_{i\in\nset_0}\]
of finite dimensional invariant subspaces $V_i\subseteq\rset[x_1,\dots,x_n]$ which cover all $\rset[x_1,\dots,x_n]$ by \Cref{thm:gCharac} (iv).
Then the following are equivalent:
\begin{enumerate}[(i)]
\item $(e^{tA})_{t\geq 0}$ is a $K$-positivity preserving semi-group, i.e.,
\[e^{tA}\pos(K)\subseteq\pos(K)\]
for all $t\geq 0$.

\item For each $i\in\nset_0$, there exists an $\varepsilon_i>0$ such that
\[ (\one - \lambda A)^{-1} (\pos(K)\cap V_i)\ \subseteq\ \pos(K)\cap V_i\]
for all $\lambda\in [0,\varepsilon_i)$.
\end{enumerate}
\end{thm}

The proof is an adaption of the proof of \cite[Prop.\ 6.2]{didio25KPosPresGen} and also an adaption of \cite[Prop.\ 2.3]{ouhabaz05}.

\begin{proof}
(i) $\Rightarrow$ (ii): 
Let $i\in\nset_0$, $p\in \pos(K)\cap V_i$, and let $\|\,\cdot\,\|_i$ be a norm on $V_i$ and also the corresponding operator norm on $V_i$.
For $\lambda > \|A|_{V_i}\|_i$ and all $t\geq 0$,
\[ \left\| e^{-\gamma t}\cdot e^{t A} p\right\|_i \leq \|p\|_i\cdot e^{(\|A|_{V_i}\|_i - \gamma)\cdot t}\]
and hence
\[ \gamma\cdot(\gamma\cdot\one - A)^{-1} p = \gamma\cdot \int_0^\infty e^{-\gamma t}\cdot e^{tA} p~\diff t.\]
Therefore, since
\[\gamma\cdot \int_0^\infty e^{-\gamma t}~\diff t = 1\]
and $e^{-\gamma t}$ can be approximated by step functions, we conclude that
\[\gamma\cdot (\gamma\cdot\one - A)^{-1} p\]
is a convex linear combination of elements
\[e^{t A} p\]
in the closed convex set $\pos(K)\cap V_i$, i.e.,
\[\gamma\cdot (\gamma\cdot\one - A)^{-1} (\pos(K)\cap V_i)\ \subseteq\ \pos(K)\cap V_i\]
for all $\gamma> \|A|_{V_i}\|_i$.
With $\varepsilon_i := \|A|_{V_i}\|_i^{-1}$ and hence
\[\lambda := \gamma^{-1}\in [0,\varepsilon_i)\]
we get the assertion (ii) for every $i\in\nset_0$.

(ii) $\Rightarrow$ (i):
Let $t\geq 0$ and $p\in\pos(K)\cap V_i$ for some $i\in\nset_0$.
Then, by \Cref{cor:taylor},
\[e^{tA} p = e^{t A|_{V_i}} p = \lim_{k\to\infty} \left( \one - \frac{t A|_{V_i}}{k} \right)^{-k} p \in \pos(K)\cap V_i\]
since
\[ (\one - \lambda A)^{-1} (\pos(K)\cap V_i)\ \subseteq\ \pos(K)\cap V_i\]
for all $k > t\cdot\varepsilon_i$ by assumption (ii).
The limit $k\to\infty$ of
\[\left( \one - \frac{t A|_{V_i}}{k}\right)^{-k} p\]
is in $\pos(K)\cap V_i$, since $\pos(K)$ is closed in the LF-topology of $\rset[x_1,\dots,x_n]$ and $V_i$ is closed since it is finite dimensional.

Since $p\in\pos(K)$ and $t\geq 0$ were arbitrary, (i) is proved.
\end{proof}

The resolvent description is a complete description.
But it is very abstract and infinitely many conditions have to be checked.
Additionally, for
\[ [0,\varepsilon_i)\]
the same restrictions as in \cite[Rem.\ 6.4]{didio25KPosPresGen} hold, i.e., $[0,\varepsilon_i)$ can in general not be extended and
\[\lim_{i\to\infty} \varepsilon_i = 0.\]

\section{Making Non-Negative Polynomials into Sums-of-Squares}
\label{sec:nonnegativeSoS}

We now come to making non-negative polynomials into sums of squares.
At first we move a set $S\subseteq\rset[x_1,\dots,x_n]$ to some cone $C\subseteq\rset[x_1,\dots,x_n]$.

\begin{thm}\label{thm:main}
Let $n\in\nset$, $d\in\nset_0$, $S\subseteq\rset[x_1,\dots,x_n]_{\leq d}$, and let $C\subseteq\rset[x_1,\dots,x_n]_{\leq d}$ be a full dimensional cone such that there is a linear functional
\[l:\rset[x_1,\dots,x_n]_{\leq d}\to\rset \quad\text{with}\quad l(p)>0 \quad\text{for all}\quad p\in \left(\overline{\cone S}\cup C\right)\setminus\{0\}.\]
If $f\in\inter C$ and the linear operator $A$ is defined by
\[A:\rset[x_1,\dots,x_n]_{\leq d}\to\rset[x_1,\dots,x_n]_{\leq d},\quad p\mapsto Ap := l(p)\cdot f,\]
then there exists a constant $\tau = \tau(f,C,S,l)>0$ such that
\[e^{tA}S \quad\subseteq\quad C\]
for all $t\geq\tau$.
\end{thm}
\begin{proof}
Let
\[\tilde{S} := \overline{\cone \conv S}\]
be the closed convex cone generated by $S$.
Since $l(p)>0$ for all $p\in \overline{S}$, $\tilde{S}$ is closed and pointed, i.e., $\tilde{S}$ has a compact base $B$.

Let $p\in B$ and $f\in\inter C$, i.e., $l(p)>0$ and $l(f)>0$.
Then
\[Ap = l(p)\cdot f \qquad\text{and}\qquad A^k p = l(p)\cdot l(f)^{k-1}\cdot f\]
for all $k\in\nset$.
Hence,
\begin{multline}\label{eq:etA}
e^{tA} p = \sum_{k\in\nset_0} \frac{t^k}{k!}\cdot A^k p\\
= p + l(p)\cdot\sum_{k\in\nset} \frac{t^k}{k!}\cdot l(f)^{k-1}\cdot f
= p + \frac{l(p)}{l(f)}\cdot \left(e^{t\cdot l(f)}-1\right)\cdot f
\end{multline}
and
\[\frac{l(f)}{l(p)}\cdot \left(e^{t\cdot l(f)}-1\right)^{-1}\cdot e^{tA}p \quad\xrightarrow{t\to\infty}\quad f.\]
Since
\[\frac{l(f)}{l(p)}\cdot \left(e^{t\cdot l(f)}-1\right)^{-1}\cdot e^{tA}p\]
is continuous in $t$ and $p$ and since $f$ is an inner point of $C$, there exists a neighborhood $U\subseteq C$ of $f$ and a constant $\tau\geq 0$ such that
\[e^{tA}S \quad\subseteq\quad e^{tA}\tilde{S} \quad\subseteq\quad \cone U \quad\subseteq\quad C\]
for all $t\geq\tau$.
\end{proof}

In the previous theorem it is clear that the main space $\rset[x_1,\dots,x_n]_{\leq d}$ can be replaced by any Banach space $\cV$.
The construction of the operator $A$ in the previous result is very specific.
This allows for the following \emph{memory} and \emph{complexity analysis}.

\begin{rem}[Memory and complexity analysis of \Cref{thm:main}.]\label{rem:complexity}
Set
\[N := \dim\rset[x_1,\dots,x_n]_{\leq d} = \binom{n+d}{n}.\]
By (\ref{eq:etA}), the transformation
\[T:\rset[x_1,\dots,x_n]_{\leq d}\to\rset[x_1,\dots,x_n]_{\leq d}\qquad\text{with}\qquad TS\subseteq C\]
from \Cref{thm:main} is given by
\[T := e^{\tau A} = e^{\tau\cdot f\cdot l} = \id + f\cdot \frac{1}{l(f)}\cdot (c-1)\cdot l\]
with
\[c := e	^{\tau\cdot l(f)}.\]
Hence, in any basis $\cB$ of $\rset[x_1,\dots,x_n]_{\leq d}$, the corresponding matrix
\[\tilde{T}\in\rset^{N\times N}\]
of $T$ can be saved by
\[\tilde{T} \overset{\wedge}{=} (\tilde{f},c,\tilde{l})\in\rset^{2N+1},\]
where $\tilde{f}\in\rset^N$ is the coefficient column vector of
\[l(f)^{-1}\cdot f\]
in the basis $\cB$ and $\tilde{l}\in\rset^N$ is the row vector representation of $l$ in the basis $\cB$, i.e., the memory requirement drops from $N^2$ to $2N+1$.

Let $M\in\rset^{N\times N}$ be an arbitrary matrix.
Then
\[M\tilde{T} = M + M\cdot \tilde{f}\cdot (c-1)\cdot \tilde{l}\]
requires at most
\[N^2 + N + N^2 = 2N^2 + N\]
multiplications and at most
\[1 + N\cdot (N-1) + N^2\]
additions, i.e., $M\tilde{T}$ \emph{requires at most $4N^2+1$ operations}!
Similarly,
\[\tilde{T}M\]
also \emph{requires at most $4N^2+1$ operations}.

To calculate $\tilde{T}^{-1}$, by (\ref{eq:etA}), we have
\[\left( e^{\tau A}\right)^{-1} = e^{-\tau A} = \id + f\cdot \frac{1}{l(f)}\cdot (c^{-1}-1)\cdot l,\]
i.e.,
\[\tilde{T}^{-1} \overset{\wedge}{=} (\tilde{f},c^{-1},\tilde{l})\in\rset^{2N+1}.\]
Hence, the calculation of $\tilde{T}^{-1}\in\rset^{N\times N}$ requires \emph{one (!) operation}, a division.
Also,
\[M\tilde{T}^{-1} \qquad\text{and}\qquad \tilde{T}^{-1}M\]
require at most $4N^2+1$ operations.

For the transformations
\[Tp\]
of a polynomial $p\in\rset[x_1,\dots,x_n]_{\leq d}$, we find from
\[Tp = p + f\cdot \frac{1}{l(f)}\cdot (c-1)\cdot l(p)\]
that the coefficient vector $\tilde{p}$ of $p$ transforms by
\[\tilde{T}\tilde{p} = \tilde{p} + \tilde{f}\cdot (c-1)\cdot \tilde{l}\tilde{p}.\]
This requires at most
\[N+1+N=2N+1\]
multiplications and at most
\[N-1 + 1 + N = 2N\]
additions, i.e., in total at most $4N+1$ operations.
Similarly,
\[T^{-1}p\]
requires at most $4N+1$ operations or $4N+2$, if the one operation for the inversion $T^{-1}$ of $T$ is included.
\exmsymbol
\end{rem}

We now apply \Cref{thm:main} to non-negative polynomials and sums of squares.

\begin{cor}\label{cor:makingSoS}
Let $n\in\nset$, $d\in\nset_0$, $K\subseteq\rset^n$ be closed with $\inter K\neq \emptyset$, and let $S\subseteq\pos(K)_{\leq 2d}$.
Then the following statements hold:
\begin{enumerate}[(i)]
\item There exists a linear operator
\[A:\rset[x_1,\dots,x_n]_{\leq 2d}\to\rset[x_1,\dots,x_n]_{\leq 2d}\]
and a constant $\tau\geq 0$ such that
\[e^{tA}S \quad\subseteq\quad \sum\rset[x_1,\dots,x_n]_{\leq d}^2\]
for all $t\geq\tau$.

\item Let
\[f\in\inter \sum\rset[x_1,\dots,x_n]_{\leq d}^2\]
and let
\[l:\rset[x_1,\dots,x_n]_{\leq 2d}\to\rset,\quad l(p):= \int_K p(x)\cdot e^{-x^2}~\diff x.\]
Then
\[A:\rset[x_1,\dots,x_n]_{\leq 2d}\to\rset[x_1,\dots,x_n]_{\leq 2d}, \quad p\mapsto Ap := l(p)\cdot f\]
is an operator as in (i).

\item Let
\[L:\rset[x_1,\dots,x_n]_{\leq 2d}\to\rset\]
be a linear functional and let $A$ be an operator as in (i).
Set
\[\tilde{S} := e^{\tau A}S\subseteq\sum\rset[x_1,\dots,x_n]_{\leq d}^2\]
and
\[\tilde{L}:\rset[x_1,\dots,x_n]_{\leq 2d}\to\rset,\quad p\mapsto \tilde{L}(p) := L(e^{-\tau A}p).\]
Then
\[\inf_{p\in S} L(p) = \inf_{p\in \tilde{S}} \tilde{L}(p) \qquad\text{and}\qquad \sup_{p\in S} L(p) = \sup_{p\in \tilde{S}} \tilde{L}(p).\]
\end{enumerate}
\end{cor}
\begin{proof}
(i) + (ii): Since $C=\sum\rset[x_1,\dots,x_n]_{\leq d}^2$ is a full dimensional cone and $l(p)>0$ for all $p\in (S\cup C)\setminus\{0\}$, \Cref{thm:main} gives the assertion.

(iii): Follows from the identity $e^{-\tau A} e^{\tau A} = \one$.
\end{proof}

\begin{rem}
An equivalent results as in \Cref{cor:makingSoS} for the sums of squares cone also holds for the Waring cone by \Cref{thm:main}.
\exmsymbol
\end{rem}

Since \Cref{cor:makingSoS} contains the degree bound $\leq 2d$, it can also be interpreted to work on the homogeneous polynomials:
\[\rset[x_0,x_1,\dots,x_n]_{=2d}\quad \cong\quad \rset[x_1,\dots,x_n]_{\leq 2d}.\]

Note, the construction of the operator $A$ in \Cref{thm:main} and \Cref{cor:makingSoS} is \emph{very specific}, since $A$ is made in such a way that $f\cdot [0,\infty)$ becomes an accumulation ray, i.e., with increasing time $t$ every $e^{tA}p$ converges towards the ray spanned by $f$.
By \Cref{rem:complexity}, we get that the transformations
\[MT,\qquad TM,\qquad MT^{-1},\qquad \text{and}\qquad T^{-1}M\]
require at most $4N^2+1$ operations and
\[Tp\qquad \text{and}\qquad T^{-1}p\]
require at most $4N+1$ operations with
\[N:=\dim\rset[x_1,\dots,x_n]_{\leq d} = \binom{n+d}{n}\]
in \Cref{thm:main} and
\[N:=\dim\rset[x_1,\dots,x_n]_{\leq 2d} = \binom{n+2d}{n}\]
in \Cref{cor:makingSoS} to make non-negative polynomials in sums-of-squares.
This requires the degree bounds $\leq d$ and $\leq 2d$.

\Cref{thm:main} does not hold for compact $K\subseteq\rset^n$ without the degree bounds.
The following \Cref{thm:main4} shows that, for compact $K\subseteq\rset^n$, there exists no linear operator
\[A:\rset[x_1,\dots,x_n]\to\rset[x_1,\dots,x_n]\]
such that
\[e^A\pos(K)\ \subseteq\ \sum\rset[x_1,\dots,x_n]^2.\]
\Cref{thm:main4} is even stronger: There exists no \emph{injective linear} operator
\[T:\rset[x_1,\dots,x_n]\to\rset[x_1,\dots,x_n]\]
with
\[T\pos(K)\ \subseteq\ \sum\rset[x_1,\dots,x_n]^2.\]

\begin{thm}\label{thm:main4}
Let $n\in\nset$ and let $K\subseteq\rset^n$ be compact.
There exists no injective and linear operator
\[T:\rset[x_1,\dots,x_n]\to\rset[x_1,\dots,x_n]\]
such that
\[T\pos(K)\ \subseteq\ \sum\rset[x_1,\dots,x_n]^2.\]
\end{thm}
\begin{proof}
Assume there exists an injective and linear operator
\[T:\rset[x_1,\dots,x_n]\to\rset[x_1,\dots,x_n] \quad\text{with}\quad T\pos(K)\subseteq\sum\rset[x_1,\dots,x_n]^2.\]
Since $T$ is injective, there exist $p\in\pos(K)$ and $c>0$ with
\[c-p\in\pos(K)\qquad \text{and}\qquad \deg Tp\ >\ \deg T1.\]
Then
\[Tc = T(p + c-p) = Tp + T(c-p)\]
with
\[Tc,\ Tp,\ T(c-p) \quad\in\sum\rset[x_1,\dots,x_n]^2\]
implies
\[\deg Tc\ =\ \max\{\deg Tp,\ \deg T(c-p)\}\ \geq\ \deg Tp\ >\ \deg Tc\]
which is a contradiction.
\end{proof}

\section{An Example of a Non-Markov Evolution}
\label{sec:nonMarkov}

We discussed and treated so far semi-groups
\[\left( e^{tA} \right)_{t\geq 0}\]
with $A\in\fg$, i.e., these are homogeneous Markov processes on the polynomials.
Such processes appear e.g.\ in \cite{cuchie12,filipov16,filipov20}.
In this section we want to give a process which is not a homogeneous Markov process.
Its time evolution comes from the non-linear Burgers' equation
\begin{equation}\label{eq:burger}
\partial_t u(x,t) = u(x,t)\cdot \partial_x u(x,t).
\end{equation}
If the initial value $u_0$ fulfills
\[\supp u_0 \subseteq [a,b],\]
then also
\begin{equation}\label{eq:suppab}
\supp u(\,\cdot\,,t) \subseteq [a,b]
\end{equation}
for all times $t$ such that the classical solution exists.
Burgers' equation is a simple example to observe a finite break down in a classical solution of a partial differential equation.
For more on Burgers' equation see e.g.\ \cite{kreiss89} and \cite{evans10}.

Instead of working with operators on polynomials, for simplicity we will now work in the dual setting, i.e., operators on sequences, see \Cref{subsec:top}.
For a time-evolution
\[T_t:\cset[x_1,\dots,x_n]\to\cset[x_1,\dots,x_n]\]
we get the dual time-evolution
\[S_t := T_t^* :\cset^{\nset_0^n}\to\cset^{\nset_0^n}\]
on sequences.
Here neither $T_t$ nor $S_t$ need to be of the form
\[e^{tA} \qquad\text{nor}\qquad e^{tA^*},\]
i.e., neither $T_t$ nor $S_t$ are a homogeneous Markov process.

For Burgers' equation we have the following time-evolution of sequences.

\begin{prop}\label{lem:burgerMoments}
Let $u_0\in\cS(\rset)$.
Then, for all $p\in\nset$ and $k\in\nset_0$, the time-dependent (signed) moments
\[s_{k,p}(t) := \int_\rset x^k\cdot u(x,t)^p~\diff x\]
of the solution $u(x,t)$ of Burgers' equation (\ref{eq:burger}) are
\[s_{k,p}(t) = \sum_{i=0}^{k} \frac{s_{k-i,p+i}(0)}{i!}\cdot t^i\cdot \prod_{j=0}^{i-1} \frac{(p+j)\cdot (k-j)}{1+(p+j)^2} \quad\in\rset[t].\]
\end{prop}
\begin{proof}
For $k=0$, we have
\begin{align*}
\partial_t s_{0,p}(t) &= \partial_t \int_\rset u(x,t)^p~\diff x\\
&= -p \int_\rset u(x,t)^p \cdot \partial_x u(x,t)~\diff x.
\intertext{Since $u(\,\cdot\,,t)$ is a Schwartz function, by partial integration,}
&= p \int_\rset \partial_x [u(x,t)^p] \cdot u(x,t)~\diff x\\
&= p^2 \int_\rset u(x,t)^p\cdot \partial_x u(x,t)~\diff x\\
&= -p\cdot \partial_t s_{0,p}(t)
\end{align*}
which gives
\[\partial_t s_{0,p}(t) = 0\]
and therefore
\[s_{0,p}(t) = s_{0,p}(0).\]

For $k\geq 1$, we get the induction
\begin{align*}
\partial_t s_{k,p}(t) &= \partial_t \int_\rset x^k\cdot u(x,t)^p~\diff x\\
&= -p \int_\rset x^k\cdot u(x,t)^p\cdot \partial_x u(x,t)~\diff x\\
&= p \int_\rset \partial_x ( x^k\cdot u(x,t)^p)\cdot u(x,t)~\diff x\\
&= p\cdot k \int_\rset x^{k-1}\cdot u(x,t)^{p+1}~\diff x + p^2 \int_\rset x^k\cdot u(x,t)^p\cdot \partial_x u(x,t) ~\diff x\\
&= p\cdot k\cdot s_{k-1,p+1}(t) - p^2\cdot\partial_t s_{k,p}(t)\\
&= \frac{p\cdot k}{1+p^2}\cdot s_{k-1,p+1}(t).
\end{align*}
Solving this induction gives
\begin{align*}
s_{k,p}(t) &= s_{k,p}(0) + \frac{p\cdot k}{1+p^2} \int_0^t s_{k-1,p+1}(\tau_1)~\diff\tau_1\\
&= s_{k,p}(0) + \frac{p\cdot k}{1+p^2} \int_0^t \Bigg[ s_{k-1,p+1}(0)\\
&\qquad\qquad\qquad\qquad + \frac{(p+1)(k-1)}{1+(p+1)^2} \int_0^{\tau_1} s_{k-2,p+2}(\tau_2)~\diff\tau_2\Bigg]\diff\tau_1\\
&\ \,\vdots\\
&= \sum_{i=0}^{k} \frac{s_{k-i,p+i}(0)}{i!}\cdot t^i\cdot \prod_{j=0}^{i-1} \frac{(p+j)\cdot (k-j)}{1+(p+j)^2},
\end{align*}
which proves the statement.
\end{proof}

The (signed) moments $s_{p,t}$ of the (signed) measure(s)
\[\diff\mu_p(x) := u(x,t)^p~\diff x\]
with $p\in\nset$ are therefore polynomials in $t$ and the coefficients depend only on the initial value $u_0$.
Since for Burgers' equation
\[u_0\geq 0\]
implies
\[u(\,\cdot\,,t)\geq 0,\]
as long as the classical solution exists, the $s_{p,k}$ are in fact moments of the non-negative Borel measure(s) $\mu_p$ at least as long as the classical solution exists.
Hence, by (\ref{eq:suppab}), a $[a,b]$-moment sequence represented by an absolutely continuous representing measure with a Schwartz function density remains under Burgers' equation a $[a,b]$-moment sequence at least as long as the classical solution exists.

The following calculation shows that a finite break down as known from partial differential equations can here also be observed through the moments.

\begin{exm}\label{exm:momFinBreakDown}
For $p=1$, we have the following three explicit time-dependent moments from \Cref{lem:burgerMoments}:
\begin{align*}
\int_\rset u(x,t)~\diff x\ =\ s_{0,1}(t)\ &=\ s_{0,1}(0),\\
\int_\rset x\cdot u(x,t)~\diff x\ =\ s_{1,1}(t)\ &=\ s_{1,1}(0) + s_{0,2}(0)\cdot t,\\
\int_\rset x^2\cdot u(x,t)~\diff x\ =\ s_{2,1}(t)\ &=\ s_{2,1}(0) + s_{1,2}(0)\cdot t + \frac{2 s_{0,3}(0)}{5}\cdot t^2.
\end{align*}
For the function
\[u_0(x) := \begin{cases} 1+x & \text{for}\ x\in [-1,0],\\ 1-x & \text{for}\ x\in [0,1],\\ 0 & \text{else}
\end{cases}\]
we have
\[s_{0,1}(0) = 1,\ s_{1,1}(0)=0,\ s_{0,2}(0) = \frac{2}{3},\ s_{2,1}(0) = \frac{1}{6},\ s_{1,2}(0) = 0,\ s_{0,3} = \frac{1}{2}\]
and therefore
\begin{equation}\label{eq:negative}
L_{s(t)}((x-t)^2) = \frac{1}{6} - \frac{2}{15}t^2 \quad\xrightarrow{t\to \pm\infty}\quad -\infty.
\end{equation}
But we have $u_0\not\in\cS(\rset)$.
By using a mollifier $S_\varepsilon$, we get
\[u_0^\varepsilon := S_\varepsilon * u_0 \in C_c^\infty(\rset)\subset \cS(\rset)\]
for any $\varepsilon > 0$ and we can chose, by continuity of the $s_{p,k}(0)$ on $\varepsilon$, an $\varepsilon>0$ small ($\varepsilon\ll 1$) such that the coefficient of $t^2$ in (\ref{eq:negative}) remains negative.
Hence, the time-dependent Riesz functional
\[L_{s(t^*)}\]
evaluated at the time-dependent non-negative polynomial
\[(x-t^*)^2\]
for some $t^*$ with $|t^*|\gg 0$ is negative.
By Haviland's Theorem \cite{havila35,havila36}, it is therefore not a moment functional.
In finite time the solution
$u$
of Burgers' equation is no longer non-negative and hence broke down in finite time.\exmsymbol
\end{exm}

Note, in the previous example the break down time observed with the moments might me larger than the break down time of the classical solution.
The polynomial
\[(x-t^*)^2\]
might not be the optimal choice for finding the break down time.
But it shows the existence of a finite break down.
Additionally, after the break down of the classical solution, a measure valued solution might exists and then stops being non-negative.
These effects and observations in non-Markov processes, especially from non-linear partial differential equations, have to be studied further.

\section*{Acknowledgments}

The author thanks Markus Schweighofer for introducing him to the concept of positivity preserver and especially giving him the reference \cite{borcea11}.
The author thanks Jan Stochel and Zenon Jablonski for reading \Cref{sec:stochel} and for the discussion at the LAW25 conference, where the author was invited by Alja\v{z} Zalar to present the research \cite{didio24posPresConst,didio25KPosPresGen,didio25hadamardLanger} and the results in the present work.

\section*{Funding}

The author is supported by the Deutsche Forschungs\-gemein\-schaft DFG with the grant DI-2780/2-1 and his research fellowship at the Zukunfts\-kolleg of the University of Konstanz, funded as part of the Excellence Strategy of the German Federal and State Government.

\section*{Conflict of Interest Statement}

The author declares no conflict of interest.

\section*{Data Availability Statement}

There is no associated data.


\newcommand{\etalchar}[1]{$^{#1}$}
\providecommand{\bysame}{\leavevmode\hbox to3em{\hrulefill}\thinspace}
\providecommand{\MR}{\relax\ifhmode\unskip\space\fi MR }
\providecommand{\MRhref}[2]{%
  \href{http://www.ams.org/mathscinet-getitem?mr=#1}{#2}
}
\providecommand{\href}[2]{#2}

\end{document}